\documentclass[reqno, intlimits, sumlimits]{amsart} 

\usepackage{lmodern,bbm} %esint

\usepackage{hyperref}
\hypersetup{colorlinks=true,linkcolor=black,citecolor=black,filecolor=black,urlcolor=black}
\usepackage[left=3.2cm,right=3.2cm, top=4cm, bottom=3cm]{geometry}

\theoremstyle{plain}
\newtheorem{theorem}			{Theorem}
\newtheorem{lemma}	[theorem]	{Lemma}
                     
\theoremstyle{definition}

\newtheorem{rem}	[theorem]	{Remark}

\numberwithin{equation}{section}
\numberwithin{theorem}{section}

\newcommand{\E}{{\mathbb{E}}}
\newcommand{\N}{{\mathbb{N}}}
\newcommand{\R}{{\mathbb{R}}}
\newcommand{\Z}{{\mathbb{Z}}}
\renewcommand{\P}{{\mathbb{P}}}

\begin{document}

 \title[Reconstructing the environment seen by a RWRE with errors]{Reconstructing a recurrent random environment from a single trajectory of a Random Walk in Random Environment with errors}
\author{Jonas Jalowy}
 \address{Jonas Jalowy, Institute for Mathematical Stochastics, University of M\"unster, Germany}
  \email{jjalowy@wwu.de}
 \author{Matthias L\"owe}
 \address{Matthias L\"owe, Institute for Mathematical Stochastics, University of M\"unster, Germany.}
 \email{maloewe@math.uni-muenster.de}
 \date{\today}
 \thanks{Research was funded by the Deutsche Forschungsgemeinschaft (DFG, German Research Foundation) under Germany 's Excellence Strategy EXC 2044-390685587, Mathematics M\"unster: Dynamics-Geometry-Structure. JJ is funded by the DFG through the SPP 2265 \emph{Random Geometric Systems}}
 
\begin{abstract}
We consider one infinite path of a Random Walk in Random Environment (RWRE, for short) in an unknown environment.
This environment consists of either i.i.d.\ site or bond randomness.
At each position the random walker stops and tells us the environment it sees at the point where it is, without telling us, where it is. These observations $\chi'$ are spoiled by reading errors that occur with probability $p<1$. We show: If the RWRE is recurrent and satisfies the standard assumptions on such RWREs, 
then with probability one in the environment, the errors, and the random walk we are able reconstruct the law of the environment. For most 
situations this result is even independent of the value of $p$. 
If the distribution of the environment has a non-atomic part, we can even reconstruct the environment itself, up to translation.
\end{abstract}

\maketitle

\section{Introduction}
Random walks in a random environment have been an extremely active and challenging field of research over the past decades. We refer the interested reader to \cite{ZeitouniRWRE,BS_RWRE, biskup_RWRE} for surveys of classic results. In this note we will consider
two settings of the RWRE:
The first is Sinai's walk (cf. \cite{Solomon}, \cite{Sinai82}), a situation with site randomness. We will refer to is as situation i). To define it, let
$\omega: \Z \to (0, 1)$ be a doubly infinite sequence of random variables. The random walk $X : \N_0 \to \Z$ in the environment
$\omega$ is a Markov chain. It starts in $0$ and its distribution  $\P_\omega$ is given by the transition probabilities in $z \in \Z$ and at time $n \in \N_0$
$$
\P_\omega(X_{n+1}= z + 1|X_n = z)=1-\P_\omega(X_{n+1}= z - 1|X_n = z)=\omega(z).
$$
$\omega$ is randomly chosen according to some probability measure $\P$ and in this situation we will assume the following:
\begin{enumerate}\label{citeria}
\item $\P$ is a product measure, i.e.\ for some probability measure $\varrho$ on $(0,1)$ (equipped with the Borel $\sigma$-field) we have $\P=\varrho^{\Z}$ and we call $\P_\omega$ the quenched law of the process.
\item There exists $\kappa$ such that $0<\kappa<\omega(0)<1-\kappa$ $\varrho$-almost surely.
\item
\begin{equation}\label{eq:sigma} 0 < \sigma^2:= \E\left(\log \frac{\omega(0)}{1-\omega(0)}\right)^2 < \infty
\end{equation}
\item
\begin{equation}\label{eq:recu}
\E\left(\log \frac{\omega(0)}{1-\omega(0)}\right)=0.
\end{equation}
\end{enumerate}
It is well-known and was already shown by Solomon that this RWRE is recurrent, if and only if \eqref{eq:recu} holds. Hence we will require that our RWRE is recurrent.

The second situation is a random conductivity model as in %\cite{Kirkpatrick_RWRE} or 
\cite{biskup_RWRE}, we will refer to it as ii). Again, we take a
doubly infinite sequence of random variables $\omega=(\omega(z))_{z \in \Z}$ of i.i.d.\ random variables with values in $(0, \infty)$, where $\omega(z)$ denotes the conductance of the edge $(z,z+1)$. The random walk $X : \N_0 \to \Z$ in the environment
$\omega$ again is a Markov chain that starts in $0$ and its distribution $\P_\omega$ is given by the transition probabilities
$$
\P_\omega(X_{n+1}= z + 1|X_n = z)=1-\P_\omega(X_{n+1}= z - 1|X_n = z)=\frac{\omega(z)}{\omega(z-1)+\omega(z)}
$$
for position $z \in \Z$ and at time $n\in \N_0$. In this situation we will assume the following for the distribution $\P$ of the environment $\omega$:
\begin{enumerate}
\item $\P$ is a product measure of a probability measure $\varrho$ on $\R^+$ with the Borel $\sigma$- field. Thus $\P=\varrho^{\Z}$ and again we call $\P_\omega$ the quenched law of the process.
\item There exists $D \in (0,1)$ such that $0<D<\omega(0)<1/D$ $\varrho$-almost surely. This condition is also called {\it uniform ellipticity}.
\end{enumerate}
It is known that under these conditions the RWRE is recurrent, see \cite[p.305]{biskup_RWRE}. Furthermore, we assume that $\omega$ is truly random and not deterministic, i.e. $\varrho\neq \delta_\alpha$.

The question we want to consider in the current note is inspired by two sources. In \cite{GantertNagel} the authors ask, whether one can almost surely reconstruct the law of a random environment in Sinai's walk from the observations of one trajectory and answer this question in the affirmative. A related question for RWRE was studied by Adelman and
Enriquez in \cite{AdelmanEnriquez}.
The work in \cite{GantertNagel}, in turn, is inspired by the scenery reconstruction problem formulated by den Hollander and
Keane \cite{dHK86}. Here $\Z$ is endowed with a random coloring (the scenery), which is only observed from the unknown positions of a random walker.
The question is whether this coloring can be reconstructed almost surely. Significant progress in this problem was made by Matzinger and co-authors, see e.g. \cite{Mat99,scen3,scen4,scen6, scen7,scen5, scen9,scen8, scen1,scen2} or the survey papers \cite{dHS06},
\cite{scendeutsch}.
In particular, in \cite{scen6} the authors show that reconstruction is still possible, if we have noisy observations with a very small percentage of errors. A new, but related question was investigated  by Lakrec, who considers a scenario where entries may be systematically changed \cite{lakrec}. Similar problems in the area of coin tossing were studied in \cite{HK97} and \cite{LPP01}. In this note we will investigate the same problem of noisy observations for the question studied by Gantert and Nagel \cite{GantertNagel}
and the same problem for RWRE with bond disorder.

To be more precise, assume that the environment situation i) or in situation ii) cannot be observed directly. However we know that the above assumptions on $\omega$ are fulfilled. Let $(X_n)_{n\in \N_0}$ be one infinite path of random walk in the random environment $\omega$ which cannot be observed directly, either. Denote by $\chi'(n)$ the environment seen by the walker at time $n$. This means in situation i) we let $\chi'(n)=\omega(X_n)$ observe the probability for a jump to the right of the vertex that $X_n$ is currently at and in situation ii) the walker $X_n$ tells the conductance of the edge it crosses, hence $\chi'(n)=\omega(\min(X_{n-1},X_n))$.
Assume moreover that we are given a corrupted version $\chi=(\chi(n))_{n\in \N_0}$ of $\chi'=(\chi'(n))_{n\in \N_0}$ given by
$$
\chi(n)=\begin{cases} \chi'(n) & \mbox{with probability } 1-p\\
Y(n) & \mbox{with probability } p.
\end{cases}
$$
Here, the $(Y(n))_n$ are some i.i.d. random variables with values in $(0,1)$ in situation i) and with values in $(D, 1/D)$
in situation ii). They are independent of everything else. Moreover, $p \in [0,1)$, and also for each $n$ we decide whether we take $\chi'(n)$ or $Y(n)$ independently of everything else. In other words
\begin{equation}\label{eq:inotherwords}
\chi(n)= \chi'(n)(1-\xi(n))+ Y(n) \xi(n)
\end{equation}
and the $\xi:=(\xi(n))_n$ are i.i.d.\ Bernoulli random variables with values in $\{0,1\}$ and parameter $p$, independent of everything else. We denote by $\nu$ the distribution of $Y(0)$ and by $\nu_a$ and $\nu_{na}$ its atomic and non-atomic part, respectively.
In a similar fashion we decompose $\varrho=\varrho_a + \varrho_{na}$ into an atomic part $\varrho_a$ and a non-atomic part
$\varrho_{na}$ and our techniques will rely on the assumption that $\nu_a$ and $\varrho_a$ have disjoint support.
The question is: Can we reconstruct the law of $\omega$ or even $\omega$ by just knowing $\chi$?

The proofs will be  easier, when the distribution of $\omega$ has a non-atomic part than in the situations where there are only atoms in the distribution of the environment. However, in both cases we will be able to reconstruct the environment almost surely and we will have no restrictions on $p<1$ whatsoever, if the distribution of the environment has a non-atomic part or if we are in situation ii).  
Note that this is way less restrictive than the assumption on the error probability in \cite{scen6} and  that even for the problem of guessing a number with errors (or lies), the error probability needs to be smaller than
$1/2$, see \cite{SpencerWinkler}.
In situation i), however, we will have a condition on $p$ when the distribution of $\varrho$ is purely atomic. 
Denoting by $\mathcal{M}:=\mathcal{M}^1((0,1))$ and $\mathcal{M}':=\mathcal{M}^1((D,1/D))$, respectively, the set of probability measures on the open unit interval and on the interval $(D,1/D)$, respectively, our result reads as follows:

\begin{theorem}\label{theo:theo1}
Let $\varrho_a$ and $\nu_a$ have disjoint support, $p\in(0,1)$ be arbitrary and consider either situation i) or situation ii). In situation i), when $\varrho=\varrho_a$ is purely atomic, we additionally assume that $0<p<\frac{\kappa}{\kappa+1}$ is known. Then, almost surely in the environment, the errors, and the realization of the RWRE, we can reconstruct $\varrho$ from $\chi$.

This means there is a measurable mapping
$\mathcal{A}: (0,1)^{\N_0}\to \mathcal{M}$ (or $\mathcal{A}: (0,1)^{\N_0}\to \mathcal{M}'$ in situation ii)), such that for almost all realizations of $\xi$, $\nu^{\N_0}$-almost all realizations of
$(Y(n))_n$, $\P$-almost all $\omega$, and $\P_\omega$-all realizations of $X$ we have
\begin{equation}\label{eq:reconstruct}
\P_\omega(\mathcal{A}(\chi)= \varrho)=1.
\end{equation}
\end{theorem}

As we will see in the proof for situation i) and $\varrho=\varrho_a$, the result continues to hold for unknown $0<p<\frac{\kappa}{\kappa+1}$, if instead either the number of atoms of $\varrho_a$ or of $\nu_a$ is finite and known. In the non-atomic case, even the environment itself can be reconstructed.

\begin{theorem}\label{theo:theo3}
If $\varrho_a$ and $\nu_a$ have disjoint support, $\varrho\neq \varrho_a$ and $p$ is arbitrary, the environment can be reconstructed almost surely in the environment, the errors, and the realization of the RWRE up to translation in both, situations i) and ii). This means there is a measurable
 mapping $\mathcal{A'}: (0,1)^{\N_0}\to (0,1)^{\Z}$ (or $\mathcal{A'}: (0,1)^{\N_0}\to (D,1/D)^{\Z}$ in situation ii)) such that for almost all realizations of $\xi$, $\nu^{\N_0}$-almost all realizations of $(Y(n))_n$, $\P$-almost all $\omega$, and $\P_\omega$-almost all $X$ 
 \begin{equation}\label{eq:reconstruct2}
\P_\omega(\mathcal{A'}(\chi) \sim \omega)=1.
\end{equation}
Here for any two environments $\omega, \omega'$ the symbol $\omega \sim \omega'$ indicates that
$\omega$ and $\omega'$ agree up to translation.
\end{theorem}

\begin{rem}
Situation i) with $p=0$  is treated in \cite{GantertNagel}. However, their proof does not carry over to situation ii). In particular, Theorem \ref{theo:theo1} with $p=0$ thus gives the result of Gantert and Nagel in situation ii). 
\end{rem}

\section{Proofs}
\begin{proof}[Proof of Theorem \ref{theo:theo1} situation i)]
We divide the proof into case (a) $\varrho_a \neq \varrho$ and case (b) of a purely atomic $\varrho= \varrho_a$.

Case (a): We will adapt the strategy from \cite{GantertNagel} to our setting of corrupted observations. The set of all possible atoms of
$\varrho+\nu$ is given by the set
$$
\mathfrak{A}_{\varrho+\nu}=\{\alpha: \exists n \text{ with } \chi(n)=\chi(n+1)=\alpha\}
$$
because only atoms of $\varrho+\nu$ can appear directly one after the other in $\chi$ (and they will). First of all notice, that, if $\alpha\notin \mathfrak{A}_{\varrho+\nu}$ appears in $\chi$, we can tell with probability one, whether it is a realization of $\varrho$ or whether it is a reading error created by $Y$. Indeed, if $\alpha$ is a realization of $\varrho$, it will occur in $\chi$ infinitely often, because we assume that RWRE is recurrent. On the other hand, non-atoms of $\nu$ will occur only once in $\chi$ (if at all).
The issue is now to determine whether a point $\alpha \in  \mathfrak{A}_{\varrho+\nu}$  is an atom of $\varrho$ or of $\nu$. %Here the assumption that we know that $\nu_a$ and $\varrho_a$ have disjoint support will help.

\begin{lemma}\label{lem:centrallemma}
In the above situation there is a test that decides with probability one correctly, whether a point $\alpha \in  \mathfrak{A}_{\varrho+\nu}$  is an atom of $\varrho$ or of $\nu$.
\end{lemma}

Note that in the case of purely non-atomic $\varrho=\varrho_{na}$, the test is not necessary and the reconstruction mechanism becomes more simple. 

\begin{proof}
Since $\varrho_a \neq \varrho$, there are realizations of $\varrho$ that are not atoms of $\varrho$.
Now let us take $\alpha \in \mathfrak{A}_{\varrho+\nu}$. Then, with probability one, in the observations $\chi$, there are $\alpha_0\neq\alpha_1 \neq \alpha_2$ created by $\varrho_{na}$ such that
$(\chi(n),\chi(n+1), \chi(n+2),\chi(n+3))=(\alpha_0,\alpha_1, \alpha, \alpha_2)$ for infinitely many $n\in \N$. The advantage is: Whenever we read the sequence
$(\alpha_0,\alpha_1, \alpha, \alpha_2)$ in $\chi$, we know exactly where the RWRE is, because the positions of the $\alpha_j$ in $\Z$ are unique. Therefore (exploiting recurrence of our RWRE) we can conclude, that $\alpha$ is an atom of $\nu$, if and only if there is also $m\in \N$ such that
$(\chi(m),\chi(m+1), \chi(m+2))=(\alpha_1, \alpha, \alpha)$.
Hence we can tell the atoms of $\varrho$ from the atoms of $\nu$.
\end{proof}

From now on we will discard the atoms of $\nu$ in $\chi$. To reconstruct $\varrho$, we use a similar strategy as proposed in \cite{GantertNagel}. To make sure to just consider new observations at points we have not seen before, we want to use
pieces of the environment that are not atoms of $\varrho$ as ''markers'' to detect that we are in a new point. Let us thus consider the set
\begin{align}\label{eq:Ti}
\mathcal T&:=\{n \in \N: \chi(n-2),\chi(n-1) \notin \mathfrak{A}_{\varrho+\nu}, 
%\nonumber\\ &\qquad 
\forall m<n-2: \chi(m) \notin \{\chi(n-2), \chi(n-1)\},\nonumber \\
&\qquad \quad \exists m >n: \chi(m)= \chi(n-2) \text{ and } \chi(m+1)= \chi(n-1) \}.%\nonumber.
\end{align}
$\mathcal T$ consists of all the time points $n$, where we have been passing a ''marker'' previously, which has not been seen before. The last condition ensures that they are really part of the true environment and not of the noise $Y$.
To filter out the elements that appear due to a reading error, we introduce the set of observations made right after these markers, conditioned on that $(X_n)$ has not stepped back
\begin{eqnarray*}
\mathcal  O: =\{\chi(n) \notin \mathfrak{A}_{\nu}: n\in \mathcal T, \chi(n)\neq\chi(n-2), \exists m>n: \chi (m)=\chi(n)\}.
\end{eqnarray*}
Due to recurrence and the fact that $\chi(n)$ is not an atom the set $\mathcal O$ indeed consists of observations from the environment $\omega$ distributed according to $\varrho$. Indeed, notice that in $\mathcal O$ we rule out those elements at times $\mathcal T$ that are either realizations of one of the random variables $Y(n)$ or that are observations $\chi'(n)$ of $(X_n)$ at points to which $(X_n)$ never returns again. However, since we assumed that $(X_n)$ is recurrent the latter case will almost surely not happen. For $\P$-almost all $\omega$, $\mathcal O$ is infinite $\P_\omega$-almost surely by Borel-Cantelli and the elements of $\mathcal O$ are realizations of independent random variables $(\eta_k)_{k\in\N}$.
Then by the law of large numbers for empirical measures
$$
\frac 1 N \sum_{k=1}^N \delta_{\eta_k} \to \varrho \qquad \P_\omega \mbox {almost surely}
$$
(convergence is understood in the sense of the weak topology of measures).

Case (b):
Now we assume that $\varrho=\varrho_a$. In situation i) we can use the 
results of \cite{GantertNagel}, while in situation ii) to follow we will make use of the environment from the particle's point of view, which provides much more details than in situation i). 

The central element of the proof is again a version of Lemma \ref{lem:centrallemma}.

\begin{lemma}\label{lem:centrallemma3}
If $\varrho=\varrho_a$ in situation i), there is a test that decides with probability one 
(in the errors, in the environment $\omega$, and for $\P_\omega$-almost every realization of the walk) correctly, whether a point $\alpha \in  \mathfrak{A}_{\varrho+\nu}$  is an atom of $\varrho$ or of $\nu$.
\end{lemma}

The idea of the proof is simply that, if $\alpha \in \mathfrak{A}_\varrho$, given that we saw an $\alpha$ two steps ago we should have a larger probability to see 
an $\alpha$ again, while for $\beta \in \mathfrak{A}_\nu$ this probability should stay independent of previous observations.

\begin{proof}
Consider the asymptotic relative frequency of observing a ''double $\alpha$'' 
\begin{align}\label{eq:h}
h (\alpha\alpha):=&\lim_{n\to\infty}\frac 1 {n}\sum_{k=1}^{n} \mathbbm 1 _{(\alpha,\alpha)} \big(\chi(k-1),\chi(k)\big).
\end{align}
Note that the above limit need not exists and if it exists it could equal $0$.\footnote{We expect this not to be the case since a constant fraction $\approx\varrho(\alpha)^2>0$ of the environment has a bounded value $\alpha\in\mathfrak A_\varrho$ and $(X_k)_k$ is recurrent, but a proof would need ergodicity, which is unknown.} 
However, in both these cases we
can almost surely conclude that $\alpha \in \mathfrak{A}_\varrho$, because
for $\alpha \in \mathfrak{A}_\nu$ the limit exists almost surely and equals $p^2\nu^2(\alpha)> 0$ by the law of large numbers.

Fix $\omega\in\Omega$ and define the hitting times of reading double $\alpha$'s as $\sigma_1=\inf_{k\in\N}\{X_{k-1},X_k\in\omega^{-1}(\alpha)\}$ and $\sigma_j=\inf_{k> \sigma_{j-1}}\{X_{k-1},X_k\in\omega^{-1}(\alpha)\}$ for $j> 1$. 

If for $\alpha \in \mathfrak{A}_\varrho$ the limit \eqref{eq:h} exists and is positive, we can relate it to the relative frequency of stopping times $J_n:=|\{j: \sigma_j\le n \}|$ via
\begin{align*}
\frac{h (\alpha\alpha)}
{(1-p)^2}
=&\lim_{n\to\infty}\frac 1 {n}\sum_{k=1}^{n} \frac{\mathbbm 1_{(0,0)}(\xi(k-1),\xi(k))}{(1-p)^2}\mathbbm 1 _{(\alpha,\alpha)} \big(\omega(X_{k-1}),\omega(X_{k}))\big)\\
=&\lim_{n\to\infty}\frac{J_n}n\cdot\lim_{n\to\infty}\frac{1}{J_n}\sum_{j=1}^{J_n}\frac{\mathbbm 1_{(0,0)}(\xi(\sigma_j-1),\xi(\sigma_j))}{(1-p)^2}=\lim_{n\to\infty}\frac{J_n}n>0
\end{align*}
by the law of large numbers for the independent sequence $\xi(\sigma_j)\sim \mathsf{Ber}(p)$ (separately for even and odd indices). 

Let us now distinguish these stopping times depending on the direction of the step by \[\sigma_j^\pm=\inf_{k> \sigma^\pm_{j-1}}\{X_{k-1},X_k\in\omega^{-1}(\alpha)\text{ and } X_k-X_{k-1}=\pm 1\}.\]
 Then, by the strong markov property the sequence $\mathbbm 1_{+1}(X_{\sigma_{j}^-+1}-X_{\sigma_j^-})\sim\mathsf{Ber}(\alpha)$ for $j\in\N$ is iid (and similarly for $\sigma^+_j$). 
Moreover for $J^\pm_n:=|\{j: \sigma^\pm_j\le n \}|$ we have $J_n^++J_n^-=J_n$. Next consider
$$
\frac{h (\alpha\alpha\alpha)}{(1-p)^3}:=\lim_{n\to\infty}\frac 1 {n(1-p)^3}\sum_{k=1}^{n} \mathbbm 1 _{(\alpha,\alpha,\alpha)} \big(\chi(k-1),\chi(k),\chi(k+1)\big).
$$
Again if the limit does not exists, we conclude that 
$\alpha \in \mathfrak{A}_\varrho$.
Otherwise, for $\alpha \in \mathfrak{A}_\varrho$, by the law of large numbers and abbreviating $\tilde\xi(k)=\frac{\mathbbm 1_{(0,0,0)}(\xi(k-1),\xi(k),\xi(k+1))}{(1-p)^3}$, we have
\begin{align}
    %\frac{h (\alpha\alpha\alpha)}{(1-p)^3}=&\lim_{n\to\infty}
&\frac 1 {n}\sum_{k=1}^{n} \tilde\xi(k)\mathbbm 1 _{(\alpha,\alpha,\alpha)} \big(\omega(X_{k-1}),\omega(X_{k}),\omega(X_{k+1})\big)\nonumber\\
\ge &\frac 1 {n}\sum_{k=1}^{n} \tilde\xi(k)\mathbbm 1 _{\omega^{-1}(\alpha)} \big(X_{k-1}\big)\mathbbm 1 _{\omega^{-1}(\alpha)}\big(X_{k})\mathbbm 1 _{X_{k-1}}\big(X_{k+1}\big)\nonumber\\
=&\frac{J_n}n\cdot\frac 1 {J_n}\sum_{j=1}^{J_n} \tilde\xi(\sigma_j)\mathbbm 1_{\pm(1,-1)}\big(X_{\sigma_j}-X_{\sigma_j-1},X_{\sigma_j+1}-X_{\sigma_j} \big)\label{eq:h(aaa)}\\
\sim&\frac{h (\alpha\alpha)}{(1-p)^2}\left[ \frac{J_n^+} {J_n}\frac 1 {J_n^+}\sum_{j=1}^{J^+_n} \tilde\xi(\sigma^+_j)\mathbbm 1_{-1}\big(X_{\sigma^+_j+1}-X_{\sigma^+_j} \big)+\frac{J_n^-} {J_n}\frac 1 {J_n^-}\sum_{j=1}^{J^-_n} \tilde\xi(\sigma^-_j)\mathbbm 1_{+1}\big(X_{\sigma^-_j+1}-X_{\sigma^-_j} \big)\right] \nonumber\\
\sim &\frac{h (\alpha\alpha)}{(1-p)^2}\left[\frac{J_n^+} {J_n}(1-\alpha)+\frac{J_n^-} {J_n}\alpha\right]\ge \kappa \frac{h (\alpha\alpha)}{(1-p)^2}\nonumber
%=& 2\alpha(1-\alpha)\frac{h (\alpha\alpha)}{(1-p)^2}>\left(\alpha \kappa +(1-\alpha)\kappa\right)\frac{h (\alpha\alpha)}{(1-p)^2}=\kappa\frac{h (\alpha\alpha)}{(1-p)^2}.\nonumber
\end{align}
 even though we cannot control any of $J_n^\pm/J_n$ separately. Now set $h (\alpha\mid\alpha\alpha):= 
\frac{h (\alpha\alpha\alpha)}{h (\alpha\alpha)}\ge (1-p)\kappa$. Then, since we assume $p<\frac{\kappa}{\kappa+1}$, 
we have 
\[ h(\alpha\mid\alpha\alpha)\ge(1-p)\kappa>p\ge p \nu(\beta)=h(\beta\mid\beta\beta)\] for all $\beta\in \mathfrak A_\nu.$ Therefore we 
decide that
$\alpha \in \mathfrak A_\varrho$, iff $h(\alpha\mid\alpha\alpha) > p$. If instead the number $M$ of atoms of $\varrho_a$ (or of $\nu_a$) is known, we decide that the $M$ largest occurrences $h(\alpha\mid\alpha\alpha)$ to correspond to $\alpha\in\mathfrak A_\varrho$ (or the $M$ smallest values to belong to $\mathfrak A_\nu$ respectively).
\end{proof}

Let us now prove Theorem \ref{theo:theo1} for the purely atomic case (b) in situation i). 
We will do so by combining Lemma \ref{lem:centrallemma3} with the techniques presented in \cite{GantertNagel}.   	
Gantert and Nagel show in \cite{GantertNagel}, p. 6-7, that for each piece of uncorrupted observations (i.e.\ from
$\chi'$) one can build an empirical measure that converges to $\varrho$ almost surely, when the
length of the piece of uncorrupted observations converges to infinity.
To this end they use the idea of representing the environment as a random walk on a tree: If
$\varrho$ has $N\in[2,\infty]$ many atoms, the consider the $N$-regular tree $T_N$ with root $o$. The vertices of $T_N$ are colored by a mapping $\psi$ with the atoms of $\varrho$ in such a way, that $o$ gets the color $\omega(0)$
and that every vertex has exactly one neighbor of every color. Then the environment can be represented as a doubly infinite nearest neighbor path $\mathcal R$ in $T_N$ by requiring that we start in $o$ and $\psi(\mathcal{R}(z))=\omega(z)$ for all $z \in \Z$. This path can be constructed by 
first following the environment in the direction of the non-negative integers step by step: We start in $o$, the next point $\mathcal{R}(1)$ of the path is the (unique) neighbor of $o$ with color $\omega(1)$, the point $\mathcal{R}(2)$ is the unique neighbor of $\mathcal{R}(1)$ with color $\omega(2)$, etc. Then we do the same with the negative integers, starting in $o$ again.  
For fixed $\omega$ the observations $\chi'$ can be mapped to a path $U$ on $T_N$ by requiring that
$U(0)=o$ and that $U$ is a nearest neighbor path on $T_N$ with $\psi(U_n)=\chi'(n)$ 
for all $n \in \N$. The advantage of this construction is that we are able tell, when the RWRE enters a fresh piece of the 
environment. Indeed, whenever $\mathcal R$ is at a ''new'' point in $T_N$, also $X_n$ is at a new point in $\Z$.

The key idea in \cite{GantertNagel}  is to consider {\it straight, first} crossings of certain segments in the image of $\mathcal R$ by $U$, i.e. first crossings in the shortest possible time. 
This is best illustrated in the case that $|\mathfrak A_\varrho|=2$ (the general case is being reduced to this case). Then $\mathfrak A_\varrho=\{\alpha_1, \alpha_2\}$
with $\varrho(\alpha_i)=\beta_i, i=1,2$, $T_N=\Z$ and, without loss of generality the coloring of $\Z$ is such that 
$\psi(0)=\psi(1)=\alpha_1$ and hence $\psi(4m)=\psi(4m+1)=\alpha_1$, and $\psi(4m+2)=\psi(4m+3)=\alpha_2, m\in \Z$. Consider the 
intervals $I_m=[4m+1, 4m+4]$ in $\Z$. The probability of a straight crossing of $I_m$ when  hit for the first time by $U$ is given by %an explicit function of $\alpha_2$ 
%and $\beta_2$, namely it is 
$(1-\alpha_2(1-\alpha_2))(1-\beta_2^2)$
(cf. \cite{GantertNagel}) and the events that a straight crossing of $I_m$ occurs, when 
$U$ hits $I_m$ for the first time are independent (%among others, 
since hitting a new interval with $U$ tells that RWRE is at a ''new'' place in $\Z$). Thus, $\beta_2$ can 
be estimated by applying the law of large numbers to (a suitably rescaled version of) the number of direct crossings of the $I_m$
when $U$ hits it for the first time: If we take a time interval $[1,T]$ and consider the $M$ intervals $I_m$, which are hit for the first time in $[1,T]$ and focus on the subset of size $M'$ of these intervals, for which there is a straight crossing when the interval is hit for the first time. Then $M'/M$ is an estimator for  $(1-\alpha_2(1-\alpha_2))(1-\beta_2^2)$ from which we can deduce an estimator for $\beta_2$. When $M$ and thus $M'$ become large we have that $M'/M \to (1-\alpha_2(1-\alpha_2))(1-\beta_2^2)$ for almost all $\omega$ and almost all realizations of the RWRE. As $T\to \infty$ we also have that $M \to\infty$ for almost all $\omega$ and almost all realizations of the RWRE (the reader is referred to \cite{GantertNagel}, p.
6-7 for more details). Finally, note that the same applies, if we consider a time interval $[t,t+T-1]$ instead of $[1,T]$ since the distribution of the environment is shift-invariant. Of course, in this case we just consider $I_m$'s that are hit for the first time within $[t,t+T-1]$.

Note that we are almost surely able to detect a sequence of pieces of uncorrupted observations of increasing lengths. We are also almost surely able to  find a subsequence of this sequence in which we 
hit a strictly increasing number of $I_m$'s for the first time (relative to this sequence, i.e.\ not taking into account previous observations). 
Applying the reconstruction mechanism from \cite{GantertNagel} to the time intervals
in this subsequence, without taking into account previous observations, %them
yields a sequence of empirical measures that almost surely converges to $\varrho$.
\end{proof}

\begin{proof}[Proof of Theorem \ref{theo:theo1} situation ii)]
Now the walker tells us the conductance of the previously passed edge, which could be crossed again backwards in the next step and hence we cannot define the set of atoms  $\mathfrak A_{\varrho+\nu}$ as before. Still, we need to distinguish observations of the environment distribution $\varrho$ from those of the error distribution $\nu$.

To this end introduce the so-called ''environment viewed from the particle'' (cf. \cite[Chapter 1]{BS_RWRE}). Let
$\Omega=(D,1/D)^{\Z}$ and let $\tau_x, x \in \Z$ be the canonical shift by $x$ on $\Omega$. For $\omega \in \Omega$ consider the process
$$\overline \omega := (\overline \omega_n)_{n \in \N_0}:= (\tau_{X_n} \omega)_{n \in \N_0}.$$ Then, $(\overline \omega_n)_{n \in \N_0}$ is a Markov chain on the
state space $\Omega$ with transition probabilities
$$
R(\omega, \omega'):= \frac{\omega(0) }{\omega(0) + \omega(-1)} \delta{\tau_1 \omega}(\omega')+ 
\frac{ \omega(-1)}{\omega(0) + \omega(-1)}\delta{\tau_{-1} \omega}(\omega')
$$
(see \cite[Proposition 1.1]{BS_RWRE} or \cite[Lemma2.1]{biskup_RWRE}). Its invariant measure is absolutely continuous with respect to $\P$ 
with Radon-Nikodym derivative
\begin{equation}\label{eq:RNd} 
  \frac{d\mathbb{Q}}{d\P}(\omega)=\frac {\omega(-1)+\omega(0)} Z,
\end{equation}
where $Z$ is given by $Z:=\int(\omega(-1)+\omega(0))d\P(\omega)=2 \E (\omega(0))< \infty$
due to the uniform ellipticity condition.
Moreover, the Markov shift is ergodic for $\mathbb Q$ (see \cite[Lemma 2.1 and Propositon 2.3]{biskup_RWRE}). 

\begin{lemma}\label{lem:centrallemmaii}
In situation ii) there is a test that decides with probability one 
(in the errors, in the environment $\omega$, and for $\P_\omega$-almost every realization of the walk) correctly, whether an observation $\alpha$ is a realization of $\varrho_{na},\varrho_a,\nu_{na}$ or $\nu_a$. 
\end{lemma}

\begin{proof}
First, note that non-atoms of $\nu$ are the only observations that almost surely occur only once in $\chi$, hence we discard them immediately.

The idea of the proof is similar to the proof of Lemma \ref{lem:centrallemma3}, but now we will be able to make exact calculations. 
Among all the occurrences of an atom $\alpha\in\mathfrak A_{\varrho+\nu}$ (to be defined in \eqref{eq:Atomsii} below) in the observations, determine the empirical occurrence of another subsequent $\alpha$ by
\begin{align}\label{eq:defhii}
h(\alpha\mid\alpha)=\frac{\lim_{n\to\infty}\frac 1 {n-1} \sum_{k=1}^{n-1} \mathbbm 1 _{(\alpha, \alpha)}(\chi(k),\chi(k+1))}{\lim_{n\to\infty}\frac 1 {n-1} \sum_{k=1}^{n-1} \mathbbm 1 _{\alpha}(\chi(k))}.
\end{align}
If $\alpha\in\mathfrak A_\nu$, the existence and positivity of the limits is, a.s. ensured by the law of large numbers. An error atom $\alpha\in\mathfrak A_\nu$ appears independently of the previous observation with probability $h(\alpha\mid\alpha)=p\nu(\alpha)$.

On the other hand for an atom $\alpha\in\mathfrak A_\varrho$ of $\varrho$, the existence of $h$ will follow from an ergodicity argument. Intuitively we should expect $h(\alpha\mid\alpha)>h(\alpha)$, i.e. the empirical occurrence should be larger if we condition on the previously passed edge having conductance $\alpha$, due to the opportunity to jump back along the same edge. In particular we claim to decide $\alpha\in\mathfrak A_\nu$ iff $h(\alpha\mid\alpha)=h(\alpha)$.

To make this rigorous, recall that by ergodicity, see \cite[Equation (2.20)]{biskup_RWRE}, for any $f=f(\omega,\tilde\omega)$ with $\E _{\mathbb Q} \E_\omega|f(\omega,\tau_{X_1}\omega )|<\infty$ we have $\lim_{n\to\infty}   \frac 1 n \sum_{k=0}^{n-1} f(\bar\omega_k,\bar\omega_{k+1})=\E _{\mathbb Q}\E_\omega f(\omega ,\tau_{X_1}\omega)$ for $\P$ almost all $\omega$ and $\P_\omega$ almost all paths $(X_k)_k$. If we include another independent ergodic sequence of random variables, e.g. $\tilde\xi(k)=\frac{1-\xi(k)}{1-p}$ as below, then the tuple $(\tilde\xi,\bar\omega)$ is ergodic since the Bernoulli process $\tilde\xi$ is mixing, see \cite[Theorem 4.10.6. (7)]{dynSys}. Therefore we have the ergodic theorem
\begin{align}\label{eq:Ergodicity}
\lim_{n\to\infty}   \frac 1 n \sum_{k=0}^{n-1} \tilde\xi(k+1)f(\bar\omega_k,\bar\omega_{k+1})=\E_{\xi}(\tilde\xi(0))\E _{\mathbb Q}\E_\omega f(\omega ,\tau_{X_1}\omega)
\end{align}
almost surely in $\P_\xi,\P,\P_\omega$. Hence by choosing a particular $f$, the conditioned empirical frequencies can be computed explicitly in terms of $\mathbb Q$, $p$, and $\nu$. For any realization $\alpha$ of $\varrho$ (both atoms and non atoms), we calculate
\begin{align}
\frac{h(\alpha)}{(1-p)}=&\lim_{n\to\infty}\frac 1 {(n-1)(1-p)}\sum_{k=0}^{n-2} \mathbbm 1 _{\alpha} \big(\chi(k+1)\big)=
\lim_{n\to\infty}\frac 1 {n-1}\sum_{k=0}^{n-2} \tilde\xi(k+1)\mathbbm 1 _{\alpha} \big(\chi'(k+1)\big)\nonumber\\
=&\lim_{n\to\infty}\frac 1 {n-1}\sum_{k=0}^{n-2}\tilde\xi(k+1)\Big[ \mathbbm 1 _{\alpha} \big(\bar\omega_k(0)\big)\mathbbm 1 _{\tau_1\bar\omega_{k}}\big(\bar\omega_{k+1}\big)+\mathbbm 1 _{\alpha} \big(\bar\omega_k(-1)\big)\mathbbm 1 _{\tau_{-1}\bar\omega_{k}}\big(\bar\omega_{k+1}\big)\Big]\nonumber\\
=&\int\int 
\Big[ \mathbbm 1 _{\alpha} \big(\omega(0)\big)\mathbbm 1 _{1}(X_1)+\mathbbm 1 _{\alpha} \big(\omega(-1)\big)\mathbbm 1 _{-1}(X_1)\Big]
d\P_\omega(X) d\mathbb Q(\omega)\nonumber\\
=&  \int \mathbbm 1 _{\alpha} \big(\omega(0)\big)\P_\omega( X_1=1)d\mathbb Q(\omega)+\int\mathbbm 1 _{\alpha} \big(\omega(-1)\big)\P_\omega( X_1=-1)d\mathbb Q(\omega) \nonumber\\
=& \int \mathbbm 1 _{\alpha} \big(\omega(0)\big)\frac{\omega(0)}{\omega(0)+\omega(-1)} \frac{\omega(0)+\omega(-1)}Z d\mathbb P(\omega) \nonumber\\
&\qquad +\int\mathbbm 1 _{\alpha} \big(\omega(-1)\big)\frac{\omega(-1)}{\omega(-1)+\omega(0)} \frac{\omega(0)+\omega(-1)}Z d\mathbb P(\omega)
=2\varrho(\alpha)\frac{\alpha}Z\label{eq:h(a)ii}
\end{align}
for any $\alpha\in\mathfrak A_{\varrho}$. In the second line we split the event into $\alpha$ arising from a step to the right or left, then we apply ergodicity \eqref{eq:Ergodicity}. Accordingly, we decide $\alpha$ to be a realization of the non-atomic part of $\varrho$ iff it occurs more than once in $\chi$ but $h(\alpha)=0$ and we define the atoms by
\begin{align}\label{eq:Atomsii}
    \mathfrak A_{\varrho+\nu}:=\{\alpha: h(\alpha)>0 \}.
\end{align}
Analogously, choosing $f(\omega,\tilde\omega)=\mathbbm 1_{(\alpha,\alpha)}(\omega(0),\tilde\omega(0))$ and $\tilde\xi(k)=\frac{\mathbbm 1_{(0,0)}(\xi(k),\xi(k+1))}{(1-p)^2}$ yields
\begin{align*}
&\frac{h(\alpha\alpha)}{(1-p)^2}:=\lim_{n\to\infty}\frac 1 {(n-1)(1-p)^2}\sum_{k=0}^{n-2} \mathbbm 1 _{\alpha\alpha} \big(\chi(k+1),\chi(k+2)\big)\\
&= \int \mathbbm 1 _{\alpha,\alpha} \big(\omega(0),\omega(1)\big)\P_\omega( X_1=1,X_2=2)+\mathbbm 1 _{\alpha,\alpha} \big(\omega(-1),\omega(-2)\big)\P_\omega( X_1=-1,X_2=-2) \\
\quad&+ \mathbbm 1 _{\alpha} \big(\omega(0)\big)\P_\omega( X_1=1,X_2=0)+\mathbbm 1 _{\alpha} \big(\omega(-1)\big)\P_\omega( X_1=-1,X_2=0)d\mathbb Q(\omega)\\
&=2\varrho(\alpha)^2\frac{\alpha}{2\alpha}\frac{\alpha}{Z}+2\varrho(\alpha)\frac\alpha Z\E\Big(\frac\alpha{\alpha+\omega(0)}\Big).
\end{align*}
Therefore we end up with 
\[ h(\alpha\mid\alpha)=(1-p)\Big[\varrho(\alpha)\frac 1 2 + \E\Big(\frac{\alpha}{\alpha+\omega(0)}\Big)\Big]=(1-p)\Big[\varrho(\alpha)+\E\Big(\frac{\alpha}{\alpha+\omega(0)}\mathbbm 1_{\alpha^c}(\omega(0))\Big)\Big]\] for $\alpha\in\mathfrak A_\varrho$ and $h(\alpha)=(1-p)\varrho(\alpha)\frac{\alpha}{\E(\omega(0))}$.\footnote{Note that such terms all have an obvious interpretation: $h(\alpha)$ is proportional to its conductance $\alpha$ and to the probability $\varrho(\alpha)$ of creating exactly this conductance in $\omega$ and $h(\alpha\mid\beta)=(1-p)\varrho(\alpha)\frac{\alpha}{\alpha+\beta}$ if $\alpha, \beta\in\mathfrak A_\varrho,\alpha\neq \beta$ describes the probability of creating an $\alpha$-edge, then crossing it and reading it.} In particular their difference is given by
\begin{align*}
    \frac{\E(\omega(0))}{1-p}(h(\alpha|\alpha)&-h(\alpha))=\varrho(\alpha)\big(\alpha\varrho(\alpha)+\int_{\alpha^c}xd\varrho(x) \big) +\int_{\alpha^c}\frac{\alpha\varrho(\alpha)}{\alpha+x}\frac{\E(\omega(0))}{\varrho(\alpha)}d\varrho(x)-\alpha\varrho(\alpha)\\
    =&-\alpha\varrho(\alpha^c)\varrho(\alpha)+\int_{\alpha^c} \frac{\alpha\varrho(\alpha)}{\alpha+x}\left[x+\frac{x^2}{\alpha}+\alpha+\frac{\E(\omega(0)\mathbbm 1_{\alpha^c}(\omega(0)))}{\varrho(\alpha)}  \right]  d\varrho(x)\\
    =&\int_{\alpha^c}\frac{1}{\alpha+x}\left[x^2\varrho(\alpha)+\alpha \E(\omega(0)\mathbbm 1_{\alpha^c}(\omega(0)))\right]d\varrho(x)>0
\end{align*}
iff $\varrho\neq \delta_\alpha$ is not singular. Then, we decide $\alpha\in\mathfrak A_\varrho$ iff $h(\alpha|\alpha)>h(\alpha)$.
\end{proof}

In order to reconstruct the distribution $\varrho$, we will again consider the cases (a) $\varrho\neq\varrho_a$ and (b) $\varrho=\varrho_a$ separately, similar to situtaion i). 
Case (a): Define the set
\begin{align*}%\label{eq:T2'}
\mathcal T'&:=\{n \in \N: \chi(n-1) \notin \mathfrak{A}_{\varrho+\nu}, \chi(m) \neq \chi(n-1) \forall m<n-1, \exists m >n: \chi(m)= \chi(n-1) \}
\end{align*}
consisting of times $n$ at which we read a part of the environment that have never been seen before. Again, we consider the set of observations made right after these markers, conditioned on that $(X_n)$ has not stepped back:
\begin{align*}
\mathcal  O': =\{\chi(n) \notin \mathfrak{A}_{\nu}: n\in \mathcal T', \chi(n)\neq\chi(n-1), \exists m>n: \chi (m)=\chi(n)\}
\end{align*}
consisting of independent new conductances. Hence analogously to situation i), we enumerate $\mathcal O'$ by $\eta_1, \eta_2, \ldots $ and by the law of large numbers for the empirical measure we obtain
$\frac 1 N \sum_{k=1}^N \delta_{\eta_k} \to \varrho$, $\P_\omega$-almost surely.

Case (b): The reconstruction mechanism from situation i) case (b) cannot be applied in situation ii) (at least not without adapting the techniques of \cite{GantertNagel} to this setting, which seems difficult). Instead, we will use the technique of the ''environment viewed from the particle'' again.
Using what we observed in \eqref{eq:h(a)ii}, we can reconstruct the law of the environment via 
$h(\alpha)=2(1-p)\varrho(\alpha)\tfrac{\alpha}Z$ for any $\alpha\in\mathfrak A_{\varrho}$. Thus, we may reconstruct the the atomic part $\varrho_a$ of the distribution of the environment by $\varrho(\alpha)=\frac{Zh(\alpha)}{2\alpha(1-p)}$. Note that $p$ as well as the normalization $Z$ in general will not be known, however it is finite and simply ignoring it gives us $\varrho$ up to proportionality factor and this can be easily overcome by knowing $\sum_{\alpha\in\mathfrak A_\varrho}\varrho(\alpha)=1$ in case (b).
\end{proof}

\begin{rem}
The previous case distinction is not necessary and was made in order to recycle objects and ideas from before. Alternatively, it is possible to reconstruct $\varrho$ directly via the distribution function of $\omega(0)$ under $\P$, i.e.\ using functions of the form $ \tfrac{Z}{2 x} \mathbbm 1_{(-\infty,\alpha]}(x)$ in \eqref{eq:h(a)ii}. It is then obvious that the same approach and result holds true on the random conductance model in $\Z^d$ for any dimension $d\in\N$ as long as the ellipticity condition is satisfied so that $\bar\omega$ is ergodic with invariant distribution $\mathbb Q$. Note that the RWRE is not recurrent anymore for $d\ge 3$, but we only used recurrence to identify the non-atoms of $\nu$, hence for $d\ge 3$ we would need the additional assumption $\nu_{na}=0$. In particular for $p=0$, this answers Question 4 in \cite{GantertNagel} for a related model. However, in this note we decided to only present $d=1$ for simplicity and comparison to situation i), where higher dimensions are out of reach. 
\end{rem}

\begin{proof}[Proof of Theorem \ref{theo:theo3}]
Recall that we showed that in the proof of Theorem \ref{theo:theo1} in both situations i)
and ii) we are able to detect whether an element in our record $\chi(n)$ was produced by an observation from $\chi'(n)$ or by a random
variable $Y(n)$. 

Once we have distinguished the elements of $\chi'$ from the elements produced by $Y$ in $\chi$ we can basically follow the recipe described in \cite{GantertNagel}: Take two observations $\omega^{(1)}$ and $\omega^{(2)}$ in $\chi$ that are realizations of
$\varrho_{na}$. By recurrence and $p<1$ these will occur infinitely often in $\chi$. We will take all ''shortest crossing'' (a common concept in scenery reconstruction) between $\omega^{(1)}$ and $\omega^{(2)}$, i.e.\ we consider the following set
\begin{eqnarray*}
S_{\omega^{(1)},\omega^{(2)}}:=&& \hskip-0.4cm \{(\chi(m), \ldots, \chi(m+l)): \chi(m)=\omega^{(1)}, \chi(m+l)=\omega^{(2)} \\
&&\qquad \mbox{ and } l= \min\{l':\exists n \mbox{ such that } \chi(n)=\omega^{(1)}, \chi(n+l')=\omega^{(2)} \}\}.
\end{eqnarray*}
Then, by Borel-Cantelli, for almost all environments $\omega$ such that $(X_n)$ is recurrent on $\omega$ we have that $P_\omega$-almost surely $l$ corresponds to direct crossings from the (unique) point $z \in \Z$  with
$\omega(z)=\omega^{(1)}$ to the (unique) point $z' \in \Z$  with $\omega(z')=\omega^{(2)}$, i.e.\ $l$ is the distance between $z$ and $z'$. Depending on
the size of $l$ and the value of $p$, many of the strings $(\chi(m), \ldots, \chi(m+l)) \in S_{\omega^{(1)},\omega^{(2)}}$ may have values $\chi(n), m<n<m+l$ that are corrupted,
i.e. $\chi(n)$ was produced by $Y(n)$ rather than by $\chi'(n)$. However, we know which values in $\chi$ are produced by $\nu$ and $\varrho$, respectively. Discarding those strings from $S_{\omega^{(1)},\omega^{(2)}}$ that contain corrupted elements, we are almost surely still left with a non-empty (actually even infinite) set, that contains a piece of the random environment.
We proceed by taking a new point $\omega^{(3)}\in \chi$  that is a realization of
$\varrho_{na}$ such that $\omega^{(3)}$ is not in the already reconstructed piece of environment. We construct the corresponding set $S_{\omega^{(2)},\omega^{(3)}}$ and the resulting
piece of the random environment. If $\omega^{(1)}$ is contained in this piece, we discard this reconstruction step and take a new $\omega^{(3)}$. Otherwise we assemble the pieces together, etc. This gives a reconstruction of the environment to one side of $\omega^{(1)}$.
For the other side we do the same thing with points $\omega^{(0)}, \omega^{(-1)}, \ldots \in \chi$ such that $\omega^{(0)}, \omega^{(-1)}$ are not in the already reconstructed piece of environment. (The sense of this two-sided procedure is to avoid a description of doing a jigsaw puzzle with the reconstructed pieces of environment we would need to do, otherwise).

Finally we can decide about the orientation of the environment: By the assumption on the distribution of $\omega(0)$, for $\P$-almost all $\omega$ we will eventually have reconstructed a point in the environment $\omega(z) \neq \frac 12$ in situation i) or $\omega(z) \neq \omega(z-1)$ in situation ii) (otherwise the proofs of these two cases coincide). From here the majority of the observations (adjusted by removing corrupted observation) will walk to the right if $\omega(z) >\frac 12 $ and to the left, otherwise. This finishes the proof of Theorem \ref{theo:theo3}.
\end{proof}

\section{Concluding remarks}

\begin{rem}
\begin{enumerate}
\item Of course Theorem \ref{theo:theo3} implies Theorem \ref{theo:theo1} under the corresponding conditions.
However, the proof of Theorem \ref{theo:theo3} requires the same techniques (plus additional work) as the first.

\item Note that for the situations that $p=0$ 
Gantert and Nagel in \cite{GantertNagel} proved that the law of the environment can be reconstructed almost surely, even if RWRE is transient.

\item In situation i), similar to the observation in \cite{GantertNagel}, if $\varrho_{na} \neq 0$, we can detect from $\chi$
whether $(X_n)$ is recurrent in $\omega$ or not: No matter, whether $(X_n)$ is recurrent or not, for almost all $
\omega$, $Y$, and $\xi$ there is an $\alpha$ in the support of $\varrho_{na}+\nu_{na}$ such that there are $n\neq m \in \N$ with $\chi(n)=
 \alpha $ as well as $\chi(m)=\alpha$, because there is a positive probability of jumping backwards.
 This tells us that $\alpha$ is created from $\varrho$. Hence, 
 if $\alpha$ occurs
in $\chi$ infinitely often, $(X_n)$ is recurrent in $\omega$, otherwise it is not.

However, it seems hard to find a way to test whether RWRE is recurrent in the general setup. Indeed, in the situation of Theorem
\ref{theo:theo1}, we could detect whether RWRE is recurrent, if we knew which realizations of the non-atomic parts belong to
$\varrho$ and which of them belong to $\nu$. However, to find this out we strongly use the knowledge that RWRE is recurrent.
Of course, this entire point just applies in situation i).

\item In situation ii), the assumption that the environment should be non-deterministic ($\varrho\neq \delta_\alpha$) is necessary for the reconstruction. If all conductances are equal, then it is impossible to distinguish $\mathfrak A_\varrho$ and $\mathfrak A_\nu$ as long as $p$ is unknown.

\item At first sight it might be reasonable to expect a proof for situation i) in the same way as for situation ii): The empirical occurrences of an $\alpha$ after we have already seen some $\alpha$'s should be bigger due to the opportunity to jump back. As we saw in the proof of Lemma \ref{lem:centrallemmaii} however, this is given by a conditioned probability under the equivalent invariant measure $\mathbb Q$, which is not known to exist in situation i). Molchanov \cite[p.273-280]{Molchanov} presents two models of environment with site randomness, where $\mathbb Q$ is known: Either $\E(\tfrac {1-\omega(0)}{\omega(0)})<1$, which excludes the necessary recurrence assumption, or in a non-i.i.d.\ environment, where adjacent values depend in exactly the way handled by our situation ii). From this point of view, it is very natural to consider situation ii). 

\end{enumerate}
\end{rem}

\section*{Acknowledgments} 
We are grateful to Nina Gantert for many hints on the behavior of RWRE.
We also thank an anonymous referee and and anonymous
Associate Editor for many useful remarks that spotted a mistake in the first version and helped to improve the paper.


\begin{thebibliography}{LMM04}

\bibitem[AE04]{AdelmanEnriquez}
Omer Adelman and Nathana\"{e}l Enriquez.
\newblock Random walks in random environment: what a single trajectory tells.
\newblock {\em Israel J. Math.}, 142:205--220, 2004.

\bibitem[Bis11]{biskup_RWRE}
Marek Biskup.
\newblock Recent progress on the random conductance model.
\newblock {\em Probab. Surv.}, 8:294--373, 2011.

\bibitem[BS02a]{BS_RWRE}
Erwin Bolthausen and Alain-Sol Sznitman.
\newblock {\em Ten lectures on random media}, volume~32 of {\em DMV Seminar}.
\newblock Birkh\"{a}user Verlag, Basel, 2002.

\bibitem[BS02b]{dynSys}
Michael Brin and Garrett Stuck.
\newblock {\em Introduction to dynamical systems}.
\newblock Cambridge university press, 2002.

\bibitem[dHS06]{dHS06}
Frank den Hollander and Jeffrey~E. Steif.
\newblock Random walk in random scenery: a survey of some recent results.
\newblock In {\em Dynamics \& stochastics}, volume~48 of {\em IMS Lecture Notes
  Monogr. Ser.}, pages 53--65. Inst. Math. Statist., Beachwood, OH, 2006.

\bibitem[GN14]{GantertNagel}
Nina Gantert and Jan Nagel.
\newblock Reconstructing the environment seen by a {RWRE}.
\newblock {\em Electron. Commun. Probab.}, 19:no. 27, 9, 2014.

\bibitem[HK97]{HK97}
Matthew Harris and Michael Keane.
\newblock Random coin tossing.
\newblock {\em Probab. Theory Related Fields}, 109(1):27--37, 1997.

\bibitem[KdH86]{dHK86}
M.~Keane and W.~Th.~F. den Hollander.
\newblock Ergodic properties of color records.
\newblock {\em Phys. A}, 138(1-2):183--193, 1986.

\bibitem[L\"01]{scendeutsch}
Matthias L\"{o}we.
\newblock Rekonstruktion zuf\"{a}lliger {L}andschaften.
\newblock {\em Math. Semesterber.}, 48(1):29--48, 2001.

\bibitem[Lak19]{lakrec}
Tsviqa Lakrec.
\newblock Scenery reconstruction for random walk on random scenery systems.
\newblock {\em preprint, arXiv: 1909.07470}, 2019.

\bibitem[LM02]{scen3}
Matthias L\"{o}we and Heinrich Matzinger, III.
\newblock Scenery reconstruction in two dimensions with many colors.
\newblock {\em Ann. Appl. Probab.}, 12(4):1322--1347, 2002.

\bibitem[LM03]{scen4}
Matthias L\"{o}we and Heinrich Matzinger, III.
\newblock Reconstruction of sceneries with correlated colors.
\newblock {\em Stochastic Process. Appl.}, 105(2):175--210, 2003.

\bibitem[LM08]{scen2}
J\"{u}ri Lember and Heinrich Matzinger.
\newblock Information recovery from a randomly mixed up message-text.
\newblock {\em Electron. J. Probab.}, 13:no. 15, 396--466, 2008.

\bibitem[LMM04]{scen5}
Matthias L\"{o}we, Heinrich Matzinger, and Franz Merkl.
\newblock Reconstructing a multicolor random scenery seen along a random walk
  path with bounded jumps.
\newblock {\em Electron. J. Probab.}, 9:no. 15, 436--507, 2004.

\bibitem[LPP01]{LPP01}
David~A. Levin, Robin Pemantle, and Yuval Peres.
\newblock A phase transition in random coin tossing.
\newblock {\em Ann. Probab.}, 29(4):1637--1669, 2001.

\bibitem[Mat99]{Mat99}
Heinrich Matzinger.
\newblock Reconstructing a three-color scenery by observing it along a simple
  random walk path.
\newblock {\em Random Structures Algorithms}, 15(2):196--207, 1999.

\bibitem[Mat05]{scen9}
Heinrich Matzinger.
\newblock Reconstructing a two-color scenery by observing it along a simple
  random walk path.
\newblock {\em Ann. Appl. Probab.}, 15(1B):778--819, 2005.

\bibitem[ML06]{scen1}
Heinrich Matzinger and J\"{u}ri Lember.
\newblock Reconstruction of periodic sceneries seen along a random walk.
\newblock {\em Stochastic Process. Appl.}, 116(11):1584--1599, 2006.

\bibitem[Mol94]{Molchanov}
S~Molchanov.
\newblock Lectures on random media.
\newblock In {\em Lectures on probability theory}, pages 242--411. Springer,
  1994.

\bibitem[MR03a]{scen7}
Heinrich Matzinger and Silke W.~W. Rolles.
\newblock Reconstructing a piece of scenery with polynomially many
  observations.
\newblock {\em Stochastic Process. Appl.}, 107(2):289--300, 2003.

\bibitem[MR03b]{scen6}
Heinrich Matzinger and Silke W.~W. Rolles.
\newblock Reconstructing a random scenery observed with random errors along a
  random walk path.
\newblock {\em Probab. Theory Related Fields}, 125(4):539--577, 2003.

\bibitem[MR06]{scen8}
Heinrich Matzinger and Silke W.~W. Rolles.
\newblock Retrieving random media.
\newblock {\em Probab. Theory Related Fields}, 136(3):469--507, 2006.

\bibitem[Sin82]{Sinai82}
Ya.~G. Sina\u{\i}.
\newblock The limit behavior of a one-dimensional random walk in a random
  environment.
\newblock {\em Teor. Veroyatnost. i Primenen.}, 27(2):247--258, 1982.

\bibitem[Sol75]{Solomon}
Fred Solomon.
\newblock Random walks in a random environment.
\newblock {\em Ann. Probability}, 3:1--31, 1975.

\bibitem[SW92]{SpencerWinkler}
Joel Spencer and Peter Winkler.
\newblock Three thresholds for a liar.
\newblock {\em Combin. Probab. Comput.}, 1(1):81--93, 1992.

\bibitem[Zei04]{ZeitouniRWRE}
Ofer Zeitouni.
\newblock Random walks in random environment.
\newblock In {\em Lectures on probability theory and statistics}, volume 1837
  of {\em Lecture Notes in Math.}, pages 189--312. Springer, Berlin, 2004.

\end{thebibliography}
\end{document}